\documentclass[a4paper]{scrartcl}
\usepackage{amsmath}
\usepackage{amssymb}
\usepackage{amsthm}
\usepackage{tikz}
\usetikzlibrary{calc,decorations.markings}

\newtheorem{theorem}{Theorem}[section]

\theoremstyle{remark}
\newtheorem{remark}{Remark}[section]

\numberwithin{equation}{section}

\begin{document}
\title{Spectral Densities of Singular Values of Products of Gaussian and Truncated Unitary Random Matrices}

\author{Thorsten Neuschel
\thanks{Institut de Recherche en Math\'{e}matique et Physique, Universit\'{e} Catholique de Louvain, Chemin du
Cyclotron 2, B-1348 Louvain-La-Neuve, Belgium. E-mail: Thorsten.Neuschel@uclouvain.be}}

\maketitle

\abstract{We study the densities of limiting distributions of squared singular values of high-dimensional matrix products composed of independent complex Gaussian (complex Ginibre) and truncated unitary matrices which are taken from Haar distributed unitary matrices with appropriate dimensional growth. In the general case we develop a new approach to obtain complex integral representations for densities of measures whose Stieltjes transforms satisfy algebraic equations of a certain type. In the special cases in which at most one factor of the product is a complex Gaussian we derive elementary expressions for the limiting densities using suitable parameterizations for the spectral variable. Moreover, in all cases we study the behavior of the densities at the boundary of the spectrum.}



\section{Introduction}

In this paper we study the densities of the limiting distributions of squared singular values of products of independent random matrices of the type
\begin{equation}\label{mixedproduct}Y_{r,s}=G_r\cdots G_{s+1}T_s \cdots T_1,
\end{equation}
where the \(j\)-th factor is of dimension \((n+\nu_j)\times(n+\nu_{j-1})\) for fixed \(\nu_j\geq 0\), \(1\leq j \leq r\) and \(\nu_0 =0\). Each factor \(G_j\) is a standard complex Gaussian matrix having independent standard complex Gaussian entries (such matrices are also known as complex Ginibre matrices) and each \(T_j\) is a truncated unitary matrix taken (as the upper left block) from a Haar distributed unitary matrix \(U_j\) of dimension \(\ell_j \times \ell_j\) with \(\ell_j\geq 2n+\nu_j +\nu_{j-1}\). The squared singular values of the product in \eqref{mixedproduct} are defined as the nonnegative eigenvalues of the \(n\)-dimensional square matrix \(Y_{r,s}^{\ast}Y_{r,s}\), where \(Y_{r,s}^{\ast}\) denotes the conjugate transpose of \(Y_{r,s}\). 

In the last years many contributions were made to the study of distributions of eigenvalues and singular values of such products of random matrices, not least due to their relevance to different areas of physics, see, e.g., \cite{Ake3, Ake4, Burda1, Burda2, Forrester1, LW, Penson}. On the level of finite dimensions, recent investigations have shown that both eigenvalues and singular values exhibit the structure of determinantal point processes and the corresponding correlation kernels have been subject of intensive investigations \cite{Ake1, Ake2, CKW, KKS, KS, KZ}. Likewise, on the level of infinite dimensions, it is of much interest to derive explicit information about the limiting distributions of the eigenvalues and singular values of such products. In the latter case this leads to the study of free convolutions of probability measures introduced by Voiculescu, which do not reveal their analytic structure easily, see e.g. \cite{Bercovici}, \cite{Biane}.

It is a classical result \cite{MP} that the appropriately rescaled singular values of a standard Gaussian matrix converge weakly, almost surely, to a deterministic limit distribution. More precisely, as \(n\rightarrow \infty\), the eigenvalues of \(\frac{1}{n}G_j^{\ast}G_j\) converge weakly, almost surely, to the Marchenko-Pastur distribution on \([0,4]\) with density
\begin{equation}\label{MP}x\mapsto \frac{1}{2\pi}\frac{\sqrt{4-x}}{\sqrt{x}},
\end{equation}
which is also known as the asymptotic zero distribution of suitably rescaled classical Laguerre polynomials. Moreover, it is well known that as \(n\rightarrow \infty\) the eigenvalues of \(T_j^{\ast}T_j\) converge weakly, almost surely, to the arcsine distribution on \([0,1]\) with density
\begin{equation}\label{arcsine}x\mapsto \frac{1}{\pi}\frac{1}{\sqrt{x(1-x)}},
\end{equation}
if we determine the growth of the dimensions of \(U_j\) appropriately, for instance by the condition that \(\ell_{j}-2n\) is independent of \(n\), which is the regime we will subsequently focus on. The latter distribution is the asymptotic zero distribution of classical Jacobi polynomials rescaled onto the unit interval. Both limiting distributions are member of a two-parameter family of measures called Raney distributions denoted by \(R_{\alpha,\beta}\) with real parameters \(\alpha>1\) and \(0<\beta \leq\alpha\). The measure \(R_{\alpha,\beta}\) is compactly supported on the positive real axis and it is given by the moment sequence (also called Raney numbers)
\[R_{\alpha,\beta}(k)=\frac{\beta}{k\alpha +\beta}\binom{k\alpha +\beta}{k},\quad k\in\mathbb{N}_{0}.\]
For a recent investigation of the Raney distribution in the context of random matrix theory see \cite{Penson}, \cite{Forrester2} and the references therein. An important special case of Raney distributions are the Fuss-Catalan distributions of order \(r\) denoted by \(FC_r\), which are obtained by choosing the parameters \(\alpha=r+1\) and \(\beta=1\) for a positive integer \(r\) and the moments of \(FC_r\) are called Fuss-Catalan numbers. Investigating the moments of the limiting measures in \eqref{MP} and \eqref{arcsine} shows that the Marchenko-Pastur distribution can be identified as the Fuss-Catalan distribution of order \(1\) (or equivalently as \(R_{2,1}\)) whereas the arcsine distribution can be identified as the Raney distribution \(R_{1,\frac{1}{2}}\).

A powerful machinery to characterize limiting distributions of eigenvalues of products of random matrices is the notion of free multiplicative convolution which was developed in free probability theory (see, e.g., \cite{Speicher, Voiculescu}). Using this approach it can be shown \cite{JPM} that as \(n\rightarrow \infty\) the singular values of the rescaled product
\begin{equation}\label{scaledprod}Z_{r,s}=\frac{1}{n^{r-s}} Y_{r,s}^{\ast}Y_{r,s}
\end{equation}
converge weakly, almost surely, to a compactly supported measure \(\mu_{r,s}\) on the positive real axis given by the free multiplicative convolutions of Raney distributions
\begin{equation}\label{murs}\mu_{r,s}=R_{r-s+1,1}\boxtimes R_{1,\frac{1}{2}}^{\boxtimes s}.\end{equation}
The moments of \(\mu_{r,s}\) have recently been studied \cite{JPM} showing that they can be found explicitly in terms of Jacobi polynomials by
\[\mu_{r,s}(0) = 1\]
and for \(k\geq 1\)
\[
\mu_{r,s}(k) = \frac{1}{k2^{ks}}  P_{k-1}^{(\alpha_{k-1}, \beta_{k-1})}\left(0\right),
\]
where $P_k^{(\alpha_k, \beta_k)}(x)$ are the Jacobi polynomials with varying parameters $\alpha_k = rk + r + 1$ and $\beta_k = -(r+1-s)k - (r+2-s)$ as defined in \cite{Szego}.

However, so far the densities of \(\mu_{r,s}\) are known only in the special cases \(s=0\), \(s=1\) and \(s=r\). In the case \(s=0\) the product \eqref{mixedproduct} consists only of complex Gaussian matrices and the limiting distribution of eigenvalues of \eqref{scaledprod} is given by the Fuss-Catalan distribution of order \(r\). The corresponding density can be expressed in terms of elementary functions \cite{Bercovici}, \cite{Neu} by
\begin{equation}\label{FC}\frac{d\mu_{r,0}}{dx}(x)=\frac{(\sin\varphi)^2 (\sin r\varphi)^{r-1}}{\pi (\sin(r+1)\varphi)^{r}}
\end{equation}
where we use the following parameterization of the spectral variable \(x\)
\[x=x(\varphi)=\frac{\left(\sin{(r+1)\varphi}\right)^{r+1}}{\sin{\varphi}\left(\sin{r\varphi}\right)^r},\quad 0<\varphi<\frac{\pi}{r+1}.\]
Moreover, further representations for these densities have been found earlier in the form of Meijer G-functions \cite{Penson} and in terms of real multivariate integrals \cite{LSW}, see Remark \ref{FUSS} below.

In the case \(s=1\) the product \eqref{mixedproduct} contains \(r-1\) complex Gaussians and one truncated unitary matrix. It is known that the limiting distribution of squared singular values belongs to the Raney family as we have \cite{NS}, \cite{Ml}
\[\mu_{r,1}=R_{r,1}\boxtimes R_{1,\frac{1}{2}}=R_{\frac{r+1}{2},\frac{1}{2}},\]
so that the density allows a representation in terms of elementary functions using a suitable parameterization analogous to \eqref{FC}. However, if \(s>1\) then the limiting distributions in \eqref{murs} do not belong to the Raney family anymore and the densities are not known except in the case \(r=s\) which is connected to the global density of the Jacobi Muttalib--Borodin ensemble as recently discovered by Forrester and Wang in \cite{Forrester3} (see also Section \(3\)). 

It is the aim of this paper to derive explicit representations of the densities of \eqref{murs} in the general case and to study their behavior at the boundary of their spectrum. In these regards, in Section \(2\) we develop a new approach to derive densities of measures \(\mu\) whose Stieltjes transforms satisfy a general (algebraic) equation of the form
\[P(w)-zQ(w)=0.\]
Under suitable conditions in Theorem \ref{densitygeneral} we prove that such densities allow a complex contour integral representation of the form
\[\frac{d\mu}{dx}(x)=\frac{1}{2\pi^2 x} \Re \int\limits_{\gamma_\alpha} \log\left(1-x\frac{Q(t)}{P(t)}\right)dt,\]
where the path of integration is given as the boundary of a sector in the complex plane with opening angle \(2\alpha\).
Subsequently, in Theorem \ref{densitymurs} we use this approach to derive densities of the measures \(\mu_{r,s}\) in the cases \(r\geq s+2\) of the form
\[\frac{d\mu_{r,s}}{dx}(x)=\frac{1}{2\pi^2 x}\Re\int\limits_{\gamma_{2\pi/(r+1)}}\log\left(1-x\frac{(t-1)(t+1)^s}{t^{r+1}}\right) dt.\]
In the special case \(s=0\) this gives a new representation for the densities of the Fuss-Catalan distributions.

The remaining boundary cases \(r=s+1\) and \(r=s\) are covered in Section \(3\). We show how the method of parameterization of the spectral variable can be used to derive explicit and elementary representations for the densities of the measures \(\mu_{r,r-1}\) and \(\mu_{r,r}\). In Theorem \ref{densitymurr-1} we show that we have
\[\frac{d\mu_{r,r-1}}{dx}(x)=\frac{2^{r+1} \sin(\varphi)\left(3\sin(\varphi)-\rho_r(\varphi)\sin(2\varphi)\right)}{\pi \sin(r+1)\varphi ~\rho_r(\varphi)^{r-1}\left(4-4\rho_r(\varphi)\cos(\varphi)+\rho_r(\varphi)^2\right)},
\]
where we use the parameterization 
\[x=x(\varphi)=\frac{\rho_r(\varphi)^r \sin(r+1)\varphi}{2^r \left(3\sin(\varphi)-\rho_r(\varphi)\sin(2\varphi)\right)},\quad\quad 0<\varphi<\frac{\pi}{r+1},
\]
and the function \(\rho_r\) is defined as
\[\rho_r(\varphi)=\frac{3\sin(r\varphi)}{2\sin(r-1)\varphi}-\sqrt{\left(\frac{3\sin(r\varphi)}{2\sin(r-1)\varphi}\right)^2-\frac{2\sin(r+1)\varphi}{\sin(r-1)\varphi}}, \quad\quad 0<\varphi<\frac{\pi}{r+1}.\]

In Theorem \ref{densitymurr} we use the same approach to derive the density of \(\mu_{r,r}\) (already found in \cite{Forrester3}) which also gives an alternative proof of the global density of the Jacobi Muttalib--Borodin ensemble.

\section{Density of singular values of products of complex Gaussian and truncated unitary matrices}

In this section we study the general situation of mixed products of the type
\begin{equation}\label{mixedproducts}Y_{r,s}=G_r\cdots G_{s+1}T_s \cdots T_1,
\end{equation}
where \(r\geq s+2\geq 2\) and we ask for the density of the limiting distribution of the squared singular values \(\mu_{r,s}\) as introduced in \eqref{murs}. As the approach of finding elementary expressions for the densities by introducing suitable parameterizations of the spectral variable turns out to be appropriate only in the boundary cases \(s=0,1,r-1,r\) (see also Section \(3\)) we begin this section by developing an approach to find integral representations for densities of measures whose Stieltjes transforms satisfy certain algebraic equations.

\begin{theorem}\label{densitygeneral}Let \(\mu\) be a probability measure supported on the compact interval \([0,x^{\ast}]\) with \(x^{\ast}>0\) and let its Stieltjes transform be denoted by
\[F(z)=\int\limits_0^{x^{\ast}}\frac{1}{z-t} d\mu(t).\]
Suppose that \(w(z)=zF(z)\) is an algebraic function with a branch cut on the interval \((0,x^{\ast})\) satisfying an algebraic equation of the form
\[P(w)-zQ(w)=0,\]
where \(P\) and \(Q\) are real polynomials with gcd(P,Q)=1 and \(P(t)> 0\) for \(t\in (0,1]\), \(Q'(1)> 0\) and \(\deg P \geq \deg Q +2\) such that \(\lim_{t\rightarrow +\infty}P(t)/Q(t)=+\infty\). Moreover, suppose that for all \(x>0\) the polynomial \(w\mapsto P(w)-xQ(w)\) has exactly two roots (counted with multiplicities) inside the sector
\[S_\alpha =\{z\in \mathbb{C}~|~z=te^{is}, t\geq 0, -\alpha\leq s\leq \alpha\},\]
no roots are located on the boundary, and assume that \(P\) does not have any roots on the semi-infinite ray \(\{t e^{i \alpha}~\vert~t>0\}\)  (\(\alpha \in (0, \pi/2]\) is a fixed number).
Then the measure \(\mu\) is absolutely continuous with respect to the Lebesgue measure with a strictly positive density on \((0,x^{\ast})\) given by
\begin{equation}\label{density}\frac{d\mu}{dx}(x)=\frac{1}{2\pi^2 x} \Re \int\limits_{\gamma_\alpha} \log\left(1-x\frac{Q(t)}{P(t)}\right)dt,\end{equation}
where the path of integration \(\gamma_{\alpha}\) is given as the concatenation of two semi-infinite rays \(\gamma_{\alpha}^{(1)} \oplus \gamma_{\alpha}^{(2)}\) with
\(\gamma_{\alpha}^{(1)}\) defined as the path \(t\mapsto e^{i\alpha} /t, t> 0,\) and \(\gamma_{\alpha}^{(2)}\) is defined as the positive real axis (see Figure \ref{contourintegration}). The branch of the logarithm is chosen as the analytical continuation of the principal branch starting at the point at infinity on each ray.  
\end{theorem}

\begin{figure}[ht]
\centering
\includegraphics[scale=1]{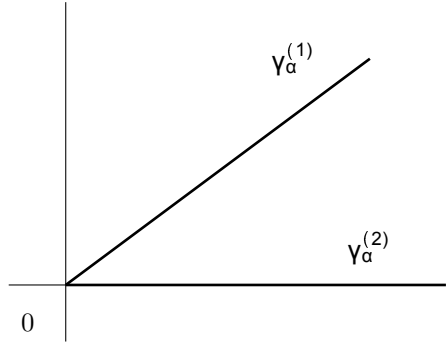}
\caption{The path of integration \(\gamma_{\alpha}\).}
\label{contourintegration}
\end{figure}

\begin{proof} We have that \(w(z)=zF(z)\) satisfies the equation \(P(w)/z-Q(w)=0\) with \(w(\infty)=1\), so by letting \(z\rightarrow \infty\) we conclude that \(Q(1)=0\). As we have \(Q'(1)> 0\) we can write \(Q(t)=(t-1)\tilde{Q}(t)\) for a polynomial \(\tilde{Q}\) with \(\tilde{Q}(1)> 0\). An application of the B\"{u}rmann-Lagrange theorem shows that \(w(z)\) is the unique solution of the equation \(P(w)-zQ(w)=0\) which is analytic at infinity with \(w(\infty)=1\) and its expansion at infinity is given by
\begin{equation}\label{expansionw}w(z)=\sum_{k=0}^{\infty}\mu_k \left(\frac{1}{z}\right)^k,~~\vert z\vert > x^{\ast},\end{equation}
where \(\mu_0=1\) and for \(k\geq 1\)
\[\mu_k=\frac{1}{k!}\frac{d^{k-1}}{dx^{k-1}}\left(\frac{P(x)}{\tilde{Q}(x)}\right)^k\Big\vert_{x=1}.\]
We will discuss the quotient \(R(t)=P(t)/Q(t)\) on the interval \((1,\infty)\). Because \(P(1)\) and \(\tilde{Q}(1)=Q'(1)\) are positive, the function \(R\) descends monotonically from \(+\infty\) near \(t=1\), and by assumption it returns to \(+\infty\) as \(t\rightarrow +\infty\). As \(w(x)\) is the unique solution of the equation \(P(w)-xQ(w)=0\) with \(w(x)\rightarrow 1\) as \(x\rightarrow \infty\), for \(t\) near \(1\) we have \(R(t)=x\) with \(t=w(x)\). Furthermore, for \(t\) near \(+\infty\) we have \(R(t)=x\) with \(t=\tilde{w}(x)\), where \(\tilde{w}\) denotes the second solution of the equation \(P(w)-xQ(w)=0\) inside the sector \(S_{\alpha}\), which has to satisfy \(\tilde{w}(x)\rightarrow +\infty\) as \(x\rightarrow +\infty\). We can conclude that the polynomial \(\tilde{Q}\) does not have any zeros inside the sector \(S_{\alpha}\), otherwise more solutions than just \(w\) and \(\tilde{w}\) could be found inside \(S_{\alpha}\) as \(x\rightarrow\infty\) (every zero of \(Q\) attracts a certain solution of the equation \(P(w)-xQ(w)=0\) as \(x\rightarrow \infty\)). The two solutions \(w(x)\) and \(\tilde{w}(x)\) have to coincide with a common value \(w^{\ast}>1\) at the branchpoint \(x=x^{\ast}\), so the derivative \(\frac{d}{dt} R(t)\) must vanish at \(t=w^{\ast}\). Moreover, using the assumption that the equation \(P(w)-xQ(w)=0\) has exactly two roots inside \(S_{\alpha}\) for all positive \(x\), we can see that \(\frac{d}{dt} R(t)\) is strictly negative on \((1,w^{\ast})\) and strictly positive on \((w^{\ast},+\infty)\). Hence, we can conclude that the quotient \(R\) is a strictly convex function on \((1,+\infty)\) with a unique minimum at \(t=w^{\ast}\) and \(x^{\ast}=R(w^{\ast})\). We can summarize the behavior of the solutions \(w(x)\) and \(\tilde{w}(x)\) on the positive real axis as follows: if we start travelling with \(x\) from \(+\infty\) along the real line towards the origin, the only solutions inside \(S_{\alpha}\) are given by \(w(x)\) and \(\tilde{w}(x)\), where \(w(x)\) emanates from \(1\) and \(\tilde{w}(x)\) emanates from \(+\infty\). As \(x\) approaches \(x^{\ast}\) from the right, \(w(x)\) strictly increases to the limit \(w^{\ast}\) whereas \(\tilde{w}(x)\) strictly decreases to the same limit. If we move \(x\) inside the cut \((0,x^{\ast})\), then the solutions \(w(x)\) and \(\tilde{w}(x)\) move to the complex plane (inside \(S_{\alpha}\)) and become complex conjugates and the sign of their imaginary parts depends on which bounday values we choose (from above or from below). At \(x=x^{\ast}\) the solutions coincide, and as we have \(P(t)/Q(t)<0\) on \((0,1)\), we can see that \(w(x)\) and \(\tilde{w}(x)\) do not coincide at any point in the interval \((0,x^{\ast})\), because for \(x\in(0,x^{\ast})\) the equation \(P(w)-xQ(w)=0\) does not have a positive solution. From this the positivity of the density of \(\mu\) on \((0,x^{\ast})\) follows.

Deriving the expression \eqref{expansionw} with respect to \(z\) we obtain
\begin{align*}w'(z)&=\frac{d}{dz} w(z) = -\frac{1}{z}\sum_{k=1}^{\infty}k\mu_k\left(\frac{1}{z}\right)^k\\
&=-\frac{1}{z}\sum_{k=1}^{\infty}\frac{1}{(k-1)!}\frac{d^{k-1}}{dx^{k-1}}\left(\frac{P(x)}{\tilde{Q}(x)}\right)^k\Big\vert_{x=1}\left(\frac{1}{z}\right)^k\\
&=-\frac{1}{z}\sum_{k=1}^{\infty}\frac{1}{2\pi i}\int\limits_{K_{\delta}(1)}\left(\frac{P(t)}{\tilde{Q}(t)}\right)^k\frac{dt}{(t-1)^k}\left(\frac{1}{z}\right)^k\\
&=-\frac{1}{z}\frac{1}{2\pi i}\int\limits_{K_{\delta}(1)}\sum_{k=1}^{\infty}\left(\frac{P(t)}{z\tilde{Q}(t)(t-1)}\right)^k dt =-\frac{1}{z}\frac{1}{2\pi i}\int\limits_{K_{\delta}(1)}\frac{dt}{1-\frac{P(t)}{\tilde{Q}(t)(t-1)z}},
\end{align*}
where \(K_{\delta}(1)\) is a positively oriented circle around \(1\) with sufficiently small radius \(\delta>0\) and the interchange of the integration with the summation is allowed because the series \(\sum_{k=1}^{\infty}\left(\frac{P(x)}{z\tilde{Q}(x)(x-1)}\right)^k\) converges uniformly with respect to \(t\) on \(K_{\delta}(1)\) if \(\vert z\vert \) is chosen sufficiently large. Thus, for large \(\vert z\vert\) we obtain
\begin{equation}\label{integralwprime}  w'(z)=\frac{1}{2\pi i}\int\limits_{K_{\delta}(1)}\frac{dt}{\frac{P(t)}{Q(t)}-z}.
\end{equation}
For large \(\vert z\vert\), the function \(w(z)\) is the only solution of the equation \(\frac{P(t)}{Q(t)}-z=0\) near \(1\) and we can replace \(K_{\delta}(1)\) by a small positively oriented circle around \(w(z)\), which we will denote by \(K(z)\):
\begin{equation}\label{integralwprime2} w'(z)=\frac{1}{2\pi i}\int\limits_{K(z)}\frac{dt}{\frac{P(t)}{Q(t)}-z}.
\end{equation}
Now we can construct a representation for the boundary values 
\[w'_+(x)=\lim_{\epsilon\rightarrow 0+} w'(x+i\epsilon) ~~~\text{and}~~~ w'_-(x)=\lim_{\epsilon\rightarrow 0+} w'(x-i\epsilon)\]
for \(x\in(0,x^{\ast})\) by means of an analytical countinuation of \eqref{integralwprime2}. We begin with \(w'_{+}(x)\) where we choose a fixed \(x\in(0,x^{\ast})\). Starting from \(z=+\infty\) we can find an analytical continuation of \eqref{integralwprime2} along the positive real line. Before we arrive at \(z=x^{\ast}\) we follow a small positively oriented semi circle around \(z=x^{\ast}\) to arrive at some point inside \((0,x^{\ast})\), from where we move to \(x\). Along this path from \(+\infty\) to \(x\) we can ensure that the circle \(K(z)\) only contains \(w(z)\) so that we arrive at
\begin{equation}\label{integralwprime3}w'_{+}(x)=\frac{1}{2\pi i}\int\limits_{K_{+}(x)}\frac{dt}{\frac{P(t)}{Q(t)}-x},
\end{equation}
where \(K_{+}(x)\) is a small positively oriented circle around \(w_{+}(x)\), which lies in the lower half-plane (as \(w(z)=zF(z)\) and \(F\) is a Stieltjes transform). In an analogous way we obtain
\begin{equation}\label{integralwprime4}w'_{-}(x)=\frac{1}{2\pi i}\int\limits_{K_{-}(x)}\frac{dt}{\frac{P(t)}{Q(t)}-x}=\overline{w'_{+}(x)},
\end{equation}
where \(K_{-}(x)\) is a small positively oriented circle around \(w_{-}(x)\) lying in the upper half-plane. In the next step we wish to make the path of integration in \eqref{integralwprime4} independent of \(x\). To this end, we observe that using Cauchy's integral theorem we can replace \(K_{-}(x)\) by the path \(\Gamma_R\) defined as concatenation of three paths
\[\Gamma_R=[0,R]\oplus\{R e^{i\varphi}~\vert~\varphi\in[0,\alpha]\}\oplus\{(R-t) e^{i\alpha}~\vert~t\in[0,R]\}.\]
Using the assumption that \(\deg P \geq \deg Q +2\) and letting \(R\rightarrow\infty\) we arrive at
\begin{equation}\label{integralwprime5}w'_{-}(x)=\frac{1}{2\pi i}\int\limits_{0}^{+\infty}\frac{dt}{\frac{P(t)}{Q(t)}-x}-\frac{1}{2\pi i}\int\limits_{0}^{e^{i\alpha}\infty}\frac{dt}{\frac{P(t)}{Q(t)}-x}=\frac{1}{2\pi i}\int\limits_{\gamma_{\alpha}}\frac{dt}{\frac{P(t)}{Q(t)}-x},
\end{equation}
because the integral over the path \(\{R e^{i\varphi}~\vert~\varphi\in[0,\alpha]\}\) vanishes as \(R\rightarrow\infty\). Moreover, for \(x\in(0,x^{\ast})\) we can explicitly find a primitive of the right hand side in \eqref{integralwprime5} so that
\begin{equation}\label{integralw1}w_{-}(x)=-\frac{1}{2\pi i}\int\limits_{\gamma_{\alpha}}\log\left(1-x\frac{Q(t)}{P(t)}\right) dt,
\end{equation}
where the contant of integration vanishes because both sides have to vanish at \(x=0\). The branch of the logarithm is defined as follows: On \(\gamma_{\alpha}^{(1)}\) we can start at the point at infinity with the principal branch of \(\log\) and on the way towards the origin we can find a continuous branch of the argument of \(1-x\frac{Q(t)}{P(t)}\) which we use to define the logarithm. On \(\gamma_{\alpha}^{(2)}\), the positive real line, we have always \(1-x\frac{Q(t)}{P(t)}>0\) so that we can use the real logarithm here. By an application of the Stieltjes inversion formula we can now compute the density of \(\mu\) on \((0,x^{\ast})\) by
\begin{align*} \frac{d\mu}{dx}(x) &=\frac{1}{2\pi i x}\left(w_{-}(x)-w_{+}(x)\right) = \frac{1}{\pi x} \Im (w_{-}(x))\\
&=\frac{1}{2\pi^2 x}\Re \int\limits_{\gamma_{\alpha}}\log\left(1-x\frac{Q(t)}{P(t)}\right) dt.
\end{align*}
\end{proof}

\begin{remark}\label{endpoints}The integral representation for the density of \(\mu\) in \eqref{density} can be used to derive the behavior of the density at the endpoints of the support \([0,x^{\ast}]\). An alternative way is the following (we use the notation from the proof of Theorem \ref{densitygeneral}): As the point \(x^{\ast}\) is a branch point connecting the two solutions \(w(z)\) and \(\tilde{w}(z)\) in a cross-wise manner, the function \(z\mapsto w(x^{\ast}+z^2)\) can be considered as a conformal mapping in a neighborhood of \(z=0\) taking the values of the branch \(w\) inside the sector 
\[\left\{z\in\mathbb{C}~\vert~z=te^{is}, t\geq 0, -\pi/2 <s<\pi/2\right\}.\] 
Hence, it has an expansion of the form
\[w(x^{\ast}+z^2)=w^{\ast}+\sum_{k=1}^{\infty} a_k z^k\]
with \(a_1\neq 0\). A computation gives for \(0<x<x^{\ast}\)
\[w_{+}(x)=\lim_{\epsilon\rightarrow 0+}\left\{w^{\ast}+\sum_{k=1}^{\infty} a_k (i\sqrt{x^{\ast}-x}+\epsilon)^k\right\}=w^{\ast}+\sum_{k=1}^{\infty} a_k (i\sqrt{x^{\ast}-x})^k\]
and
\[w_{-}(x)=\lim_{\epsilon\rightarrow 0+}\left\{w^{\ast}+\sum_{k=1}^{\infty} a_k (-i\sqrt{x^{\ast}-x}+\epsilon)^k\right\}=w^{\ast}+\sum_{k=1}^{\infty} a_k (-i\sqrt{x^{\ast}-x})^k.\]
This leads to
\begin{equation}\label{rightendpoint}\lim_{x\rightarrow x^{\ast}-}\frac{1}{\sqrt{x^{\ast}-x}}\frac{d\mu}{dx}(x)=\lim_{x\rightarrow x^{\ast}-}\frac{1}{2\pi i x}\frac{w_{-}(x)-w_{+}(x)}{\sqrt{x^{\ast}-x}}=\frac{a_1}{2\pi x^{\ast}}>0,
\end{equation}
which means that the density of \(\mu\) vanishes like a square root at the right endpoint of the support.
In contrast to that, the behavior of the density at the left endpoint of the support, which is the origin, will be determined by the order of the zero of the polynomial \(P\) at the origin. To see this, let us denote the order of the zero of \(P\) at the origin with \(\ell\), where we necessarily have \(\ell\geq 2\). The number \(\ell\) relates to the order of the branch point at the origin of the function \(w(z)\) which means that \(\ell\) branches are connected through the origin. Thus, we can consider the function \(z\mapsto w(z^{\ell})\) as a conformal mapping in the neighborhood of the origin taking the values of the branch \(w\) inside the sector \(\left\{z\in\mathbb{C}~\vert~z=te^{is}, t\geq 0, 0 <s<2\pi/\ell\right\}\). Hence, we have an expansion of the form
\[w(z^{\ell}) = \sum_{k=1}^{\infty}b_k z^k\]
with \(b_1\neq 0\). For small \(x>0\) we obtain
\[w_{+}(x)=\lim_{\epsilon\rightarrow 0+}\sum_{k=1}^{\infty}b_k \left(x^{1/\ell}+i\epsilon\right)^k=\sum_{k=1}^{\infty}b_k x^{k/\ell}\]
and
\[w_{-}(x)=\lim_{\epsilon\rightarrow 0+}\sum_{k=1}^{\infty}b_k \left(e^{2\pi i/\ell}x^{1/\ell}+\epsilon\right)^k=\sum_{k=1}^{\infty}b_k e^{2\pi ik/\ell} x^{k/\ell}.\]
This gives us
\begin{align}\label{leftendpoint}\nonumber \lim_{x\rightarrow 0+} x^{(\ell-1)/\ell}\frac{d\mu}{dx}(x)&=\lim_{x\rightarrow 0+} \frac{x^{(\ell-1)/\ell}}{2\pi i x}(w_{-}(x)-w_{+}(x))\\
&= \frac{b_1}{2\pi i}\left(e^{2\pi i/\ell}-1\right)>0,
\end{align}
which means that the density behaves like \(x^{-(\ell-1)/\ell}\) as \(x\rightarrow 0+\). 
\end{remark}
We want to point out that the constants involved in the leading terms at both endpoints of the spectrum can be obtained in form of integral representations by studying the expression \eqref{density}.

Next we want to use Theorem \ref{densitygeneral} to obtain the densities for the measures \(\mu_{r,s}\) defined in \eqref{murs} for \(r\geq s+2\). To this end, we first need to define two quantities
\begin{equation} \label{w_ast}
w_{r,s}^{\ast} := \frac{1-s + \sqrt{(1-s)^2 + 4(r+1)(r-s)}}{2(r-s)} >1
\end{equation}
and
\begin{equation} \label{x_ast}
x_{r,s}^{\ast} := \frac{r+1}{s+1} \frac{(w_{r,s}^{\ast})^r}{(w_{r,s}^{\ast} + 1)^{s-1}\left(w_{r,s}^{\ast} - \frac{s - 1}{s+1}\right)} >0.
\end{equation}

\begin{theorem}\label{densitymurs} The measure \(\mu_{r,s}\) is supported on the interval \([0,x_{r,s}^{\ast}]\) and has a strictly positive density on the interval \((0,x_{r,s}^{\ast})\) given by
\begin{equation}\label{den} \frac{d\mu_{r,s}}{dx}(x)=\frac{1}{2\pi^2 x}\Re\int\limits_{\gamma_{2\pi/(r+1)}}\log\left(1-x\frac{(t-1)(t+1)^s}{t^{r+1}}\right) dt,
\end{equation}
where the path \(\gamma_{2\pi/(r+1)}\) and the branch of the logarithm are defined as in Theorem \ref{densitygeneral}. Moreover, the density behaves like \(x^{-r/(r+1)}\) as \(x\rightarrow 0+\) and it vanishes like a square root as \(x\rightarrow x^{\ast}-\). Hence, only the boundary behavior at the origin depends on the number of matrices involved in the product \eqref{mixedproducts}.
\end{theorem}
\begin{proof} Let \(F(z)\) denote the Stieltjes transform of the measure \(\mu_{r,s}\). It can be derived using notions from free probability \cite{JPM} that the function \(w(z)=zF(z)\) satisfies the algebraic equation
\[w^{r+1}-z(w-1)(w+1)^s=0,\]
and it can be checked that \(w(z)\) has a branch cut on \((0,x_{r,s}^{\ast})\) with \(w(x_{r,s}^{\ast})=w_{r,s}^{\ast}\) as defined in \eqref{w_ast} and \eqref{x_ast}.
Setting \(P(t)=t^{r+1}\) and \(Q(t)=(t-1)(t+1)^s\) we can readily check the conditions \(gcd(P,Q)=1\), \(P(t)>0\) on \((0,1]\), \(Q'(1)>0\), \(\deg P \geq \deg Q +2\) and \(\lim_{t\rightarrow +\infty}P(t)/Q(t)=+\infty\). Moreover, \(P\) clearly has no roots on the semi-infinite ray \(\{te^{i2\pi/(r+1)}~\vert~t>0\}\). It remains to show that for all \(x>0\) the equation \(P(w)-xQ(w)=0\) has exactly two roots inside the sector \(S_{2\pi/(r+1)}\) and no roots on the boundary. To this end, we first consider the case \(s=0\) and define the functions \(f(w)=w^{r+1}-x(w-1)\) and \(g(w)=w^{r+1}+x\) for a fixed \(x>0\). We will show that \(f\) has exactly two roots inside \(\tilde{S}_{R}\) and no roots are on the boundary for sufficiently large \(R>0\), where we define
\[\tilde{S}_R=\{z\in\mathbb{C}~\vert~z=te^{is}, 0\leq t\leq R, -2\pi/(r+1) \leq s \leq 2\pi/(r+1)\}.\]
On the ray \(w=te^{i2\pi/(r+1)}\) we have by the triangle inequality
\begin{align*}\vert f(w)-g(w)\vert &= x\vert w\vert = x\vert w\vert -\vert w^{r+1}+x\vert+\vert w^{r+1}+x\vert\\
&\leq \vert w^{r+1}+x-xw\vert +\vert w^{r+1}+x\vert =\vert f(w)\vert + \vert g(w)\vert. 
\end{align*}
It can be checked by an elementary computation that this inequality is strict so we obtain 
\[\vert f(w)-g(w)\vert<\vert f(w)\vert + \vert g(w)\vert, ~w=te^{i2\pi/(r+1)}.\]
Because \(f\) and \(g\) are real polynomials we have the same inequality on the complex conjugate ray \(w=te^{-i2\pi/(r+1)}\). Moreover, choosing \(R>0\) sufficiently large, we can ensure that we always have 
\[\vert f(w)-g(w)\vert<\vert f(w)\vert + \vert g(w)\vert\]
on the entire boundary of \(\tilde{S}_{R}\). Hence, there are no roots on the boundary of \(\tilde{S}_{R}\), and by Rouch\'{e}'s theorem we can deduce that \(f(w)\) and \(g(w)\) have the same number of roots inside \(\tilde{S}_{R}\) (counted with multiplicities) for all \(R>0\) sufficiently large. As we can explicitly compute that \(g(w)\) has exactly two roots inside \(\tilde{S}_{R}\) we obtain that the same statement is true for \(f(w)\). In order to prove this statement for \(f_s(w)=w^{r+1}-x(w-1)(w+1)^s\) we now proceed inductively with respect to \(s\in\{0,\ldots,r-2\}\). So we assume the claim holds for \(f_{s-1}\) for an \(s\in\{1,\ldots,r-2\}\). On the ray \(w=te^{i2\pi/(r+1)}\) we have the inequality
\begin{align*} &\vert f_{s}(w)-f_{s-1}(w)\vert = x\vert w \vert \vert w-1\vert \vert w+1\vert^s\\
&\leq \vert w^{r+1}-x(w-1)(w+1)^{s-1}-xw(w-1)(w+1)^{s-1}\vert\\ 
&   \quad\quad+\vert w^{r+1}-x(w-1)(w+1)^{s-1}\vert\\
&=\vert f_s(w)\vert +\vert f_{s-1}(w)\vert.
\end{align*}
In this inequality we have equality if and only if we have
\begin{equation}\label{ineq1}\Im\left\{\left(w^{r+1}-x(w-1)(w+1)^{s-1}\right)\overline{w(w-1)(w+1)^{s-1}}\right\}=0
\end{equation}
and
\begin{equation}\label{ineq2}\Re\left\{\left(w^{r+1}-x(w-1)(w+1)^{s-1}\right)\overline{w(w-1)(w+1)^{s-1}}\right\}\geq0.
\end{equation}
Solving for \(x>0\) in \eqref{ineq1} gives
\[x=\frac{\Im\left\{w^{r+1}\overline{w(w-1)(w+1)^{s-1}}\right\}}{\Im\{\overline{w}\}\vert w-1\vert^2\vert w+1\vert^{2s-2}}\]
and, by replacing \(w=te^{i2\pi/(r+1)}\), from this we obtain
\[\Im\left\{e^{i2\pi/(r+1)}\left(te^{i2\pi/(r+1)}-1\right)\left(te^{i2\pi/(r+1)}+1\right)^{s-1}\right\} >0.\]
Furthermore, replacing \(x\) in \eqref{ineq2} we obtain after some computation
\[\Im\left\{\left(te^{i2\pi/(r+1)}-1\right)\left(te^{i2\pi/(r+1)}+1\right)^{s-1}\right\} \leq 0.\]
This means, that we have equality if and only if the number 
\[\left(te^{i2\pi/(r+1)}-1\right)\left(te^{i2\pi/(r+1)}+1\right)^{s-1}\] 
lies in the sector
\[\{z=te^{is}~\vert~t\geq 0, -2\pi/(r+1)\leq s\leq 0\}.\]
However, as we have \(s\leq r-2\), a careful study shows that the trace of the complex contour
\[t\mapsto \left(te^{i2\pi/(r+1)}-1\right)\left(te^{i2\pi/(r+1)}+1\right)^{s-1},~~~t\geq 0,\]
stays at a positive distance from this sector. Hence, we arrive at
\[\vert f_{s}(w)-f_{s-1}(w)\vert<\vert f_s(w)\vert +\vert f_{s-1}(w)\vert\]
on the ray \(w=te^{i2\pi/(r+1)}\) and by symmetry this remains true on the complex conjugate ray \(w=te^{-i2\pi/(r+1)}\). Again by choosing a sufficiently large \(R>0\) we can ensure that this inequality holds true for the entire boundary of \(\tilde{S}_{R}\). We conclude that \(f_s(w)\) does not have roots on the boundary and \(f_s(w)\) and \(f_{s-1}(w)\) have the same number of roots inside \(\tilde{S}_{R}\) for all \(R>0\) sufficiently large, which proves by induction that \(f_s(w)\) has exactly two roots inside \(S_{2\pi/(r+1)}\). 
Thus, all conditions of Theorem \ref{densitygeneral} are satisfied and the representation \eqref{den} follows. The behavior of the density at the endpoints of the support follows immediately from Remark \ref{endpoints}.
\end{proof}

\begin{remark} The explicit form of the densities in \eqref{den} allows us to produce plots. In Figure \ref{denr7s3plot} we see a plot of the density \(\mu_{r,s}\) in the case \(r=7\) and \(s=3\) on its support \([0,x_{7,3}^{\ast}]\), where we have
\[x_{7,3}^{\ast}=\frac{2 (w_{7,3}^{\ast})^7}{(w_{7,3}^{\ast}+1)^2 (w_{7,3}^{\ast}-\frac{1}{2})}\approx 2.015\]
with
\[w_{7,3}^{\ast}=\frac{\sqrt{33}-1}{4}.\]
In Figure \ref{denr7s3plot2} we see a plot of this density in a neighborhood of the right endpoint, which indicates that it vanishes like a square root.

\end{remark}


\begin{figure}[ht]
\centering

\includegraphics[scale=0.5]{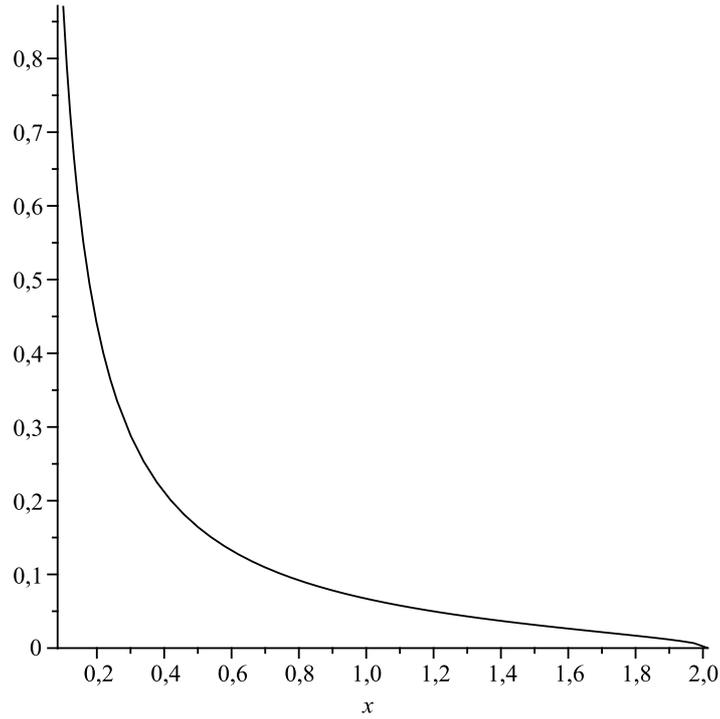}
\caption{Plot of \(\mu_{7,3}\) on its entire support.}
\label{denr7s3plot}
\end{figure}

\begin{figure}[ht]
\centering
\includegraphics[scale=0.5]{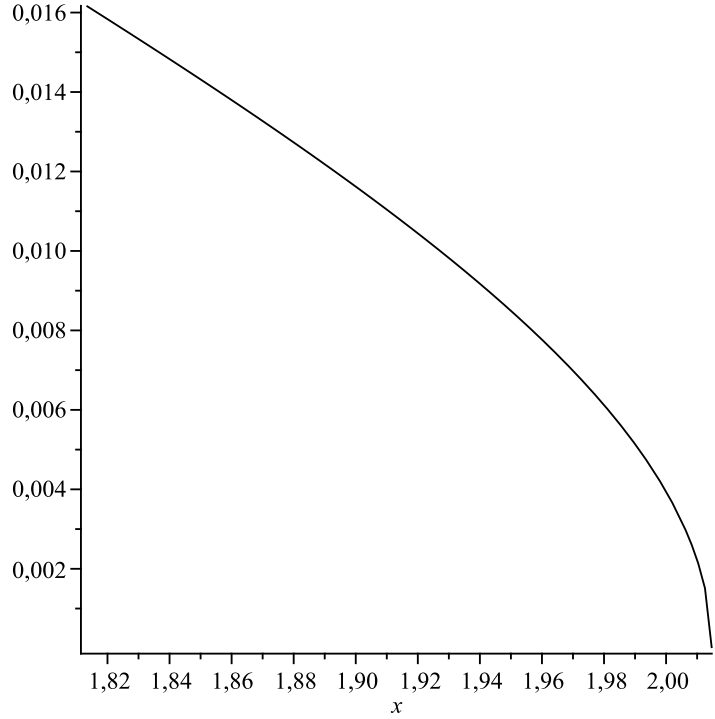}
\caption{Plot of \(\mu_{7,3}\) in the neighborhood of the right endpoint of its support.}
\label{denr7s3plot2}
\end{figure}

\begin{remark}\label{FUSS} The special case \(s=0\) in Theorem \ref{densitymurs} gives a new representation for the densities of the Fuss-Catalan distributions of order \(r\)
\[\frac{d\mu_{r,0}}{dx}(x)=\frac{1}{2\pi^2 x}\Re\int\limits_{\gamma_{2\pi/(r+1)}}\log\left(1-x\frac{t-1}{t^{r+1}}\right) dt,\quad\quad 0<x<\frac{(r+1)^{r+1}}{r^r},\]
which does not make use of a parameterization of the spectral variable (in contrast to the representation in \eqref{FC}). This representation can be seen as a companion to the representation in terms of Meijer G-functions found in \cite{Penson}. Moreover, both of these representations can be interpreted as one-dimensional complex analoga to the real multivariate integral representation found in \cite{LSW} 
\[\frac{d\mu_{r,0}}{dx}(x)=\frac{1_{(0,K]}(x)}{B\left(\frac{1}{2}, \frac{1}{2}-\frac{1}{r}\right)}\int\limits_{[0,1]^r} \frac{(\tau K-x)^{1/r-1/2}}{\sqrt{x}(\tau K)^{1/r}}F(t_1,\ldots,t_r)1_{\{\tau k \geq 0\}} ~d^r(t), \]
where \(K=\frac{(r+1)^{r+1}}{r^r}\), \(\tau=\prod\limits_{j=1}^r t_j\), \(B(a,b)\) denotes the Betafunction and
\[F(t_1,\ldots,t_r)=\frac{t_1^{1/(r+1)-1}(1-t_1)^{(r-1)/(2r+2)-1} \prod\limits_{j=2}^r t_j^{j/(r+1)-1} (1-t_j)^{j/(r(r+1))-1} }{B\left(\frac{1}{r+1}, \frac{r-1}{2r+2}\right) \prod\limits_{j=2}^rB\left(\frac{j}{r+1}, \frac{j}{r(r+1)}\right)}.\]

\end{remark}

\section{Density of singular values of products with at most one complex Gaussian matrix}

In this section we first consider products of the form
\[Y_r = G_r T_{r-1}\cdots T_{1},\]
where \(r>1\), \(G_r\) is a complex Ginibre matrix and the matrices \(T_j\) are truncated Haar distributed unitary matrices with the same conditions on the truncations as in the preceeding sections.
In order to describe the spectral density of \(\mu_{r,r-1}\) as introduced in \eqref{murs} we define the following quantities

\[c_r=\frac{3r-\sqrt{r^2+8}}{2(r-1)}\in(1,2),\]

\[x_{r,r-1}^{\ast}=\frac{c_r^{r+1}}{2^r\left(c_r-1\right)\left(2-c_r\right) }>0,\]
and
\[\rho_r(\varphi)=\frac{3\sin(r\varphi)}{2\sin(r-1)\varphi}-\sqrt{\left(\frac{3\sin(r\varphi)}{2\sin(r-1)\varphi}\right)^2-\frac{2\sin(r+1)\varphi}{\sin(r-1)\varphi}}, \quad\quad 0<\varphi<\frac{\pi}{r+1}.\]

\begin{theorem}\label{densitymurr-1} The measure \(\mu_{r,r-1}\) is supported on the interval \(\left[0,x_{r,r-1}^{\ast}\right]\) and has the density 
\begin{equation}\label{densmurr-1} \frac{d\mu_{r,r-1}}{dx}(x)=\frac{2^{r+1} \sin(\varphi)\left(3\sin(\varphi)-\rho_r(\varphi)\sin(2\varphi)\right)}{\pi \sin(r+1)\varphi ~\rho_r(\varphi)^{r-1}\left(4-4\rho_r(\varphi)\cos(\varphi)+\rho_r(\varphi)^2\right)},
\end{equation}
where 
\begin{equation}\label{para}x=x(\varphi)=\frac{\rho_r(\varphi)^r \sin(r+1)\varphi}{2^r \left(3\sin(\varphi)-\rho_r(\varphi)\sin(2\varphi)\right)},\quad\quad 0<\varphi<\frac{\pi}{r+1}.
\end{equation}
\end{theorem}
\begin{proof} Let \(F\) denote the Stieltjes transform of \(\mu_{r,r-1}\)
\[F(z)=\int\limits_0^{a}\frac{1}{z-t} d\mu_{r,r-1}(t),\]
where \(a>0\) is the finite right endpoint of the support of \(\mu_{r,r-1}\). Using properties of the \(S\)-transform from free probability theory it can be derived that the function \(w(z)=zF(z)\) satisfies the algebraic equation
\begin{equation}\label{alg3} w^{r+1}-z(w-1)(w+1)^{r-1}=0.
\end{equation}
Studying the branch points of the algebraic function given by \eqref{alg3} shows that \(w\) has a branch cut on the interval \((0,a)\) with
\[a=x_{r,r-1}^{\ast}\]
and \(w\) is the unique solution analytic on \(\mathbb{C}\cup{\{\infty\}}\backslash[0,x_{r,r-1}^{\ast}]\) taking the value \(1\) at infinity. If we again define the function \(v\) by
\[v(z)=\frac{2w(z)}{w(z)+1},\]
then \(v\) is a further analytic function on \(\mathbb{C}\cup{\{\infty\}}\backslash[0,x_{r,r-1}^{\ast}]\) taking the value \(1\) at infinity (it follows from equation \eqref{alg3} that \(w\) never takes the value \(-1\)). Moreover, as we have 
\begin{equation}\label{w}w(z)=\frac{v(z)}{2-v(z)},
\end{equation}
from \eqref{alg3} we see that \(v\) satisfies the algebraic equation
\begin{equation}\label{alg4} v^{r+1}+2^{r}z(v-1)(v-2)=0.
\end{equation}
It can be checked computationally that, like \(w\), \(v\) has the a branch cut on \((0,x_{r,r-1}^{\ast})\), and we are interested in finding the boundary values of \(v\) on the cut explicitly by introducing a suitable parameterization of the spectral variable. To this end, we first observe that for \(z=x>x_{r,r-1}^{\ast}\) the equation \eqref{alg4} has exactly two positive solutions. One solution tends to \(1\) as \(x\rightarrow +\infty\), so this solution is given by \(v\), and the second solution tends to \(2\) as \(x\rightarrow +\infty\). The two solutions approach each other if \(x\) starts travelling from \(+\infty\) towards \(x_{r,r-1}^{\ast}\) and they meet at \(x=x_{r,r-1}^{\ast}\) both taking the value \(c_r\). If we move with \(x\) inside the interval \((0,x_{r,r-1}^{\ast})\) then the solutions move away from the real axis and become complex conjugates. In order to describe them we make the ansatz \(v=\rho_r(\varphi)e^{i\varphi}\) with the positive function \(\rho_r(\varphi)>0\) to be determined. Substituting \(v=\rho_r(\varphi)e^{i\varphi}\) into \eqref{alg4} and taking the imaginary parts gives
\[\rho_r(\varphi)^{r+1}\sin(r+1)\varphi + x2^r\left(\rho_r(\varphi)^2 \sin(2\varphi)-3 \rho_r(\varphi)\sin(\varphi)\right)=0.\]
After dividing by \(\rho_r(\varphi)\) and solving for \(x\) we obtain
\begin{equation}\label{x}x=\frac{\rho_r(\varphi)^{r} \sin(r+1)\varphi}{2^r \left(3 \sin(\varphi)-\rho_r(\varphi) \sin(2\varphi)\right)}.
\end{equation}
Moreover, substituting \(v=\rho_r(\varphi)e^{i\varphi}\) into \eqref{alg4} and taking the real parts gives
\[\rho_r(\varphi)^{r+1}\cos(r+1)\varphi + x2^r\left(\rho_r(\varphi)^2 \cos(2\varphi)-3 \rho_r(\varphi)\cos(\varphi)+2\right)=0.\]
After replacing \(x\) using \eqref{x}, dividing by \(\rho_r(\varphi)^{r}\), rearranging the terms and using some standard trigonometric identities we arrive at the quadratic equation
\[\rho_r(\varphi)^2-\frac{3\sin(r\varphi)}{\sin(r-1)\varphi}\rho_r(\varphi) +\frac{2\sin(r+1)\varphi}{\sin(r-1)\varphi}=0,\]
which can be solved for \(\rho_r(\varphi)\) by
\[\rho_r(\varphi)=\frac{3\sin(r\varphi)}{2\sin(r-1)\varphi}-\sqrt{\left(\frac{3\sin(r\varphi)}{2\sin(r-1)\varphi}\right)^2-\frac{2\sin(r+1)\varphi}{\sin(r-1)\varphi}}.\]
This function is well defined for \(0<\varphi<\frac{\pi}{r+1}\) and we choose this solution of the quadratic equation as it starts with the value \(c_r\) for \(\varphi \rightarrow 0\) and vanishes as \(\varphi\rightarrow \frac{\pi}{r+1}\) (in contrast to the second solution). We use this explicit form for \(\rho_r(\varphi)\) and \eqref{x} in order to define the parameterization \eqref{para}, which is a strictly decreasing function on \((0,\frac{\pi}{r+1})\) starting at \(x_{r,r-1}^{\ast}\) and ending at \(0\). Hence, in these coordinates we can explicitly find the boundary values of \(v\) on the cut \((0,x_{r,r-1}^{\ast})\) by
\[v_{+}(x)=\rho_r(\varphi)e^{-i\varphi}\]
and 
\[v_{-}(x)=\rho_r(\varphi)e^{i\varphi},\]
where \(x=x(\varphi)\) is given by \eqref{x}. Now, in regards of \eqref{w}, by Stieltjes inversion we obtain for the density
\begin{align*}\frac{d\mu_{r,r-1}}{dx}(x(\varphi))=&\frac{1}{2\pi i x(\varphi)}\left(w_{-}(x(\varphi))-w_{+}(x(\varphi))\right)\\
=&\frac{1}{2\pi i x(\varphi)}\left(\frac{v_{-}(x(\varphi))}{2-v_{-}(x(\varphi))}-\frac{v_{+}(x(\varphi))}{2-v_{+}(x(\varphi))}\right)\\
=&\frac{1}{\pi i x(\varphi)}\frac{v_{-}(x(\varphi))-v_{+}(x(\varphi))}{\vert 2-v_{-}(x(\varphi))\vert^2}\\
=&\frac{2 \rho_r(\varphi) \sin(\varphi)}{\pi x(\varphi) \vert 2-v_{-}(x(\varphi))\vert^2},
\end{align*}
which after some further simplification leads to \eqref{densmurr-1}.
\end{proof}
\begin{remark} The expression in \eqref{densmurr-1} can be used to study the behavior of the density at the boundary of the support in an analogous way as in Remark \ref{boundaryrr}. It turns out that we have at the left endpoint of the support
\[\frac{d\mu_{r,r-1}}{dx}(x)\sim ax^{-r/(r+1)}, \quad x\rightarrow 0+,\]
and at the right endpoint of the support 
\[\frac{d\mu_{r,r-1}}{dx}(x)\sim b~ \sqrt{1-\frac{x}{x_{r,r-1}^{\ast}}}, \quad x\rightarrow x_{r,r-1}^{\ast},\]
with positive constants \(a\) and \(b\), which can be found explicitly. However, here we forgo the details of the derivation and the specification of these constants as their explicit forms turn out to be rather cumbersome. 
\end{remark}
\begin{remark} Figure \ref{densmurr-1345} shows the plots of the densities for \(\mu_{r,r-1}\) for \(r=3,4,5\) (from right to left).
\begin{figure}[!ht]
\centering
\includegraphics[scale=0.5]{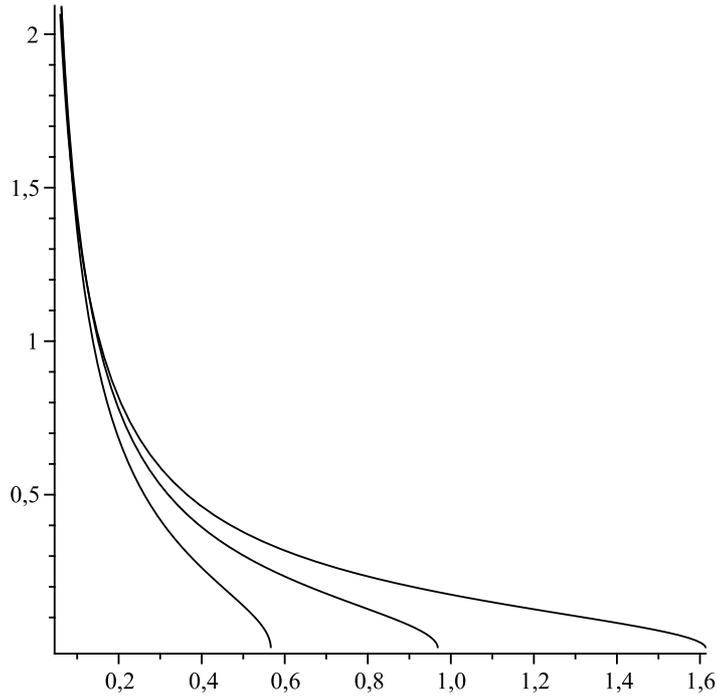}
\caption{Densities of \(\mu_{r,r-1}\) for \(r=3,4,5\) (from right to left).}
\label{densmurr-1345}
\end{figure}
\end{remark}

Next we investigate the densities of the measures \(\mu_{r,s}\) defined in \eqref{murs} in the case \(s=r\). Hence, we investigate the weak almost sure limit of the eigenvalues of the product
\[\left(T_r \cdots T_1\right)^{\ast}\left(T_r \cdots T_1\right)\]
in the regime described below \eqref{arcsine} given by the \(r\)-fold free multiplicative convolution of the arcsine measure on \([0,1]\), which is the Raney distribution \(R_{1,\frac{1}{2}}\). Thus, we have
\[\mu_{r,r}=R_{1,\frac{1}{2}}^{\boxtimes r}.\]

Recently, it was found by Forrester and Wang \cite{Forrester3} that the measure \(\mu_{r,r}\) coincides with the global distribution of the Jacobi Muttalib--Borodin ensemble and its density can be derived using information about the corresponding kernel. However, in the following theorem we rederive the density by working it out directly from the algebraic equation satisfied by its Stieltjes transform.

\begin{theorem}\label{densitymurr} The measure \(\mu_{r,r}\) is supported on the interval \(\left[0,\frac{(r+1)^{r+1}}{2^{r+1}r^r}\right]\) and has the density
\begin{equation}\label{densmurr}\frac{d\mu_{r,r}}{dx}\left(x\right)=\frac{2^{r+2} \sin(\varphi)^2 \sin(r\varphi)^{r+1}}{\pi \sin((r+1)\varphi)^r \left(4\sin(\varphi)^2\sin(r\varphi)^2+\sin((r-1)\varphi)^2\right)}
\end{equation}
where 
\[x=x(\varphi)=\frac{\sin((r+1)\varphi)^{r+1}}{2^{r+1} \sin(\varphi)\sin(r\varphi)^r},~~0<\varphi<\frac{\pi}{r+1}.\]
\end{theorem}
\begin{proof} Let \(F\) denote the Stieltjes transform of \(\mu_{r,r}\)
\[F(z)=\int\limits_0^{x^{\ast}}\frac{1}{z-t} d\mu_{r,r}(t),\]
where \(x^{\ast}>0\) is the right endpoint of the support of \(\mu_{r,r}\) (as the measure is the free multiplicative convolution of compactly supported measures on the positive real axis, \(x^{\ast}\) has to be finite). Using properties of the \(S\)-transform from free probability theory it can be derived (see, e.g., \cite{JPM}) that the function \(w(z)=zF(z)\) satisfies the algebraic equation
\begin{equation}\label{alg1} w^{r+1}-z(w-1)(w+1)^r=0.
\end{equation}
Studying the branch points of the algebraic function given by \eqref{alg1} shows that \(w\) has a branch cut on the interval \((0,x^{\ast})\) with
\[x^{\ast}=\frac{(r+1)^{r+1}}{2^{r+1}r^r}\]
and \(w\) is the unique solution analytic on \(\mathbb{C}\cup{\{\infty\}}\backslash[0,x^{\ast}]\) taking the value \(1\) at infinity. However, if we define the function \(v\) by
\[v(z)=\frac{2w(z)}{w(z)+1},\]
then \(v\) is a further analytic function on \(\mathbb{C}\cup{\{\infty\}}\backslash[0,x^{\ast}]\) taking the value \(1\) at infinity (it follows from equation \eqref{alg1} that \(w\) never takes the value \(-1\)). Moreover, as we have 
\[w(z)=\frac{v(z)}{2-v(z)},\]
from \eqref{alg1} we see that \(v\) satisfies the algebraic equation
\begin{equation}\label{alg2} v^{r+1}-2^{r+1}z(v-1)=0.
\end{equation}
Up to a scaling in the argument, this is the equation for the Stieltjes transforms in the case \(s=0\), which coincides with the Fuss-Catalan case. It is known (see, e.g., \cite{Forrester2, Neu}) that the boundary values of \(v\) on the branch cut \((0,x^{\ast})\) can be stated explicitly by
\[v_{+}(x)=\frac{\sin(r+1)\varphi}{\sin(r\varphi)}e^{-i\varphi}\]
and
\[v_{-}(x)=\frac{\sin(r+1)\varphi}{\sin(r\varphi)}e^{i\varphi},\]
if we choose the parameterization 
\[x=x(\varphi)=\frac{\sin((r+1)\varphi)^{r+1}}{2^{r+1} \sin(\varphi)\sin(r\varphi)^r},~~0<\varphi<\frac{\pi}{r+1}.\]
Hence, in these coordinates, by means of the Stieltjes inversion formula we obtain for the density
\begin{align*}\frac{d\mu_{r,r}}{dx}\left(x\right)=&\frac{1}{2\pi i x(\varphi)}\left(w_{-}(x(\varphi)-w_{+}(x(\varphi))\right)\\
=&\frac{1}{2\pi i x(\varphi)}\left(\frac{v_{-}(x(\varphi))}{2-v_{-}(x(\varphi))}-\frac{v_{+}(x(\varphi))}{2-v_{+}(x(\varphi))}\right)\\
=&\frac{1}{\pi x(\varphi)}\Im\left(\frac{\sin((r+1)\varphi)e^{i\varphi}}{2\sin(r\varphi)-\sin((r+1)\varphi)e^{i\varphi}}\right),
\end{align*}
which gives \eqref{densmurr} after some simplification.
\end{proof}
\begin{remark} In the special case \(r=1\) the density in \eqref{densmurr} reduces to the well-known arcsine measure on \([0,1]\). 
\end{remark}
\begin{remark} \label{boundaryrr}The behavior at the endpoints of the support can be derived from \eqref{densmurr} like in \cite{Forrester1}, Corollary 2.5. It turns out that we have
\[\frac{d\mu_{r,r}}{dx}\left(x\right)\sim\frac{\sin\frac{\pi}{r+1}}{\pi}x^{-r/(r+1)},~~x\rightarrow 0+,\]
and, provided that \(r>1\),
\[\frac{d\mu_{r,r}}{dx}\left(x\right)\sim\frac{2^{r+2+1/2}}{\pi}\frac{r^{r+1/2}}{(r+1)^{r+1/2}(r-1)^2}\sqrt{1-\frac{2^{r+1}r^r}{(r+1)^{r+1}}x},~~x\rightarrow \frac{(r+1)^{r+1}}{2^{r+1}r^r}-.\]
\end{remark}
\begin{remark} We want to point out an interesting relation between the spectral distribution \(\mu_{r,r}\) and the asymptotic distribution of zeros of Jacobi-Pi\~{n}eiro polynomials for large multi-indices on the diagonal, which has been found recently in \cite{NVA}. To this end, let us denote the Stieltjes transform of the limiting distribution of zeros rescaled in such a way that it is supported on \(\left[0,\frac{(r+1)^{r+1}}{r^r}\right]\) by \(G(z)\). As we explicitly know the moments of this measure, we can write
\begin{align*}G(z)=&\sum_{k=0}^\infty \binom{(r+1)k}{k} \frac{1}{z^{k+1}}\\
=&\frac{1}{z}\ _r F_{r-1}\left(\frac{1}{r+1}, \frac{2}{r+1}, \ldots, \frac{r}{r+1}; \frac{1}{r}, \frac{2}{r}, \ldots, \frac{r-1}{r} \big\vert \frac{(r+1)^{r+1}}{r^r z}\right),
\end{align*}
where \(_r F_{r-1}\) is the standard notation for generalized hypergeometric functions. It is known that such hypergeometric functions are algebraic (see, e.g., \cite{Pere}) and it can be shown that \(G\) satisfies the algebraic equation
\[\left(zG(z)\right)^{r+1}- \frac{r^r}{(r+1)^{r+1}} z\left(zG(z)-1\right)\left(zG(z)+\frac{1}{r}\right)^r=0.\]
This equation is of a similar type as \eqref{alg1}, which enables us to find a functional relation between \(G\) and the Stieltjes transform \(F\) of \(\mu_{r,r}\) in terms of a rational transformation
\[F(z)=\frac{(r+1)2^r G\left(2^{r+1}z\right)}{1+(r-1)2^r z G\left(2^{r+1}z\right)}.\]
This relation gives an alternative way to derive the corresponding densities from each other. For instance, recalling from \cite{NVA}, Theorem 1.1, that the density \(w_{JP}\) of the asymptotic zero distribution of the Jacobi-Pi\~{n}eiro polynomials on \(\left[0,\frac{(r+1)^{r+1}}{r^r}\right]\) is given by
\begin{align*}&w_{JP}(\hat{x}(\varphi))\\
&=\frac{r+1}{\pi \hat{x}(\varphi)} \frac{\sin\varphi \sin r\varphi \sin(r+1)\varphi}{(r+1)^2 \sin^2 r\varphi -2r(r+1)\sin(r+1)\varphi \sin r\varphi \cos \varphi + r^2 \sin^2(r+1)\varphi},\end{align*}
with
\[\hat{x}(\varphi)=2^{r+1}x(\varphi)=\frac{\sin((r+1)\varphi)^{r+1}}{\sin(\varphi)\sin(r\varphi)^r},~~0<\varphi<\frac{\pi}{r+1},\]
by Stieltjes inversion we obtain
\begin{align*}&\frac{d\mu_{r,r}}{dx}(x(\varphi))=\frac{1}{2\pi i}\left(F_{-}(x(\varphi))-F_{+}(x(\varphi))\right)\\
&=\frac{1}{2\pi i}\left(\frac{(r+1)2^r G_{-}\left(\hat{x}(\varphi)\right)}{1+(r-1)2^r x(\varphi) G_{-}\left(\hat{x}(\varphi)\right)}-\frac{(r+1)2^r G_{+}\left(\hat{x}(\varphi)\right)}{1+(r-1)2^r x(\varphi) G_{+}\left(\hat{x}(\varphi)\right)}\right)\\
&=(r+1)2^{r}\left\vert\frac{2(r+1)\sin r\varphi -2 r \sin(r+1)\varphi~ e^{i\varphi}}{2(r+1)\sin r\varphi -(r+1)\sin(r+1)\varphi ~e^{i\varphi}} \right\vert^2 ~\frac{G_{-}\left(\hat{x}(\varphi)\right)-G_{+}\left(\hat{x}(\varphi)\right)}{2\pi i}\\
&=(r+1)2^{r}\left\vert\frac{2(r+1)\sin r\varphi -2 r \sin(r+1)\varphi~ e^{i\varphi}}{2(r+1)\sin r\varphi -(r+1)\sin(r+1)\varphi ~e^{i\varphi}} \right\vert^2 ~w_{JP}(\hat{x}(\varphi)).
\end{align*}
In the above derivation we additionally used that
\[\hat{x}(\varphi)G_{-}\left(\hat{x}(\varphi)\right)=\frac{\sin(r+1)\varphi~e^{i\varphi}}{(r+1)\sin(r\varphi)-r\sin(r+1)\varphi~e^{i\varphi}}\]
and it can be verified by further computation that the result agrees with the formula found in \eqref{densmurr}.
\end{remark}
\begin{remark} Figure \ref{densmurr345} shows the plots of the densities for \(\mu_{r,r}\) for \(r=3,4,5\) (from right to left).
\begin{figure}[!ht]
\centering
\includegraphics[scale=0.5]{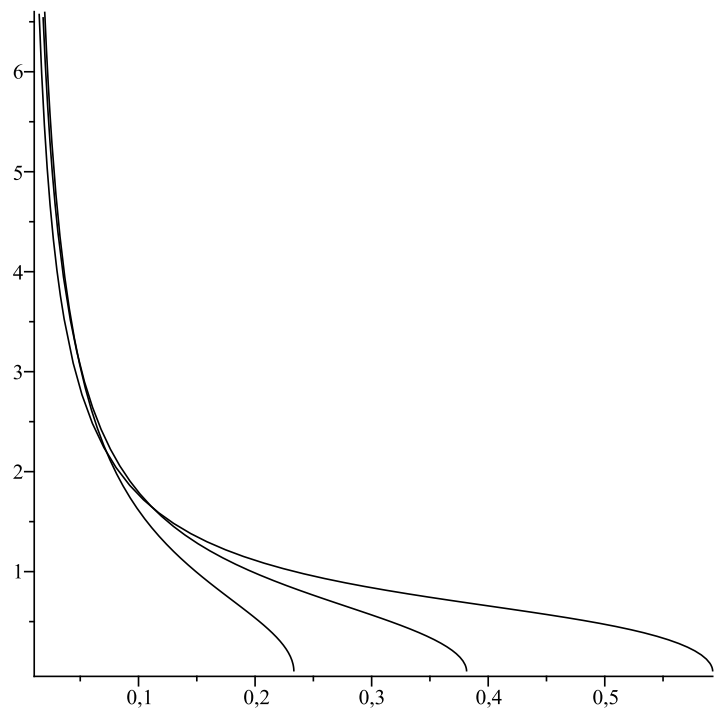}
\caption{Densities of \(\mu_{r,r}\) for \(r=3,4,5\) (from right to left).}
\label{densmurr345}
\end{figure}
\end{remark}

\section*{Acknowledgements} Thorsten Neuschel is a Research Associate with the FRS-FNRS (Belgian National Scientific Research Fund).

\end{document}